\documentclass[12pt]{article}

\usepackage{color}
\usepackage{graphics,amsmath,amssymb}
\usepackage{amsthm}
\usepackage{amsfonts}
\usepackage{latexsym}
\usepackage{url}
\usepackage{float}
\usepackage{epsf}
\usepackage[linkcolor=webgreen,filecolor=webbrown,citecolor=webgreen]{hyperref}
\usepackage[all,knot]{xy}
\xyoption{arc}

\graphicspath{{figures/}}

\newtheorem{theorem}{Theorem}

\newtheorem{lemma}{Lemma}
\newtheorem{conjecture}{Conjecture}
\newtheorem{corollary}{Corollary}

\newcommand{\lr}{\left(}
\newcommand{\rr}{\right)}
\begin{document}
\def\lr{\left(}
\def\lc{\left\{}
\def\rc{\right\}}
\def\rr{\right)}
\def\op{\text{op*}_{n,k}(\rho)}
\def\opmn{\text{op*}_{m+n,k}(\rho)}
\def\opm{\text{op*}_{m,k}(\rho)}
\def\oop{\text{op}_{n,k}(\rho)}
\def\oopn{\text{op}_{n,k}^{1/n}(\rho)}
\def\opn{\text{op*}_{n,k}^{1/n}(\rho)}

\begin{center}
\vskip 1cm{\LARGE\bf Pattern Avoidance in Ordered Set Partitions
}
\vskip 1cm
\large
Anant Godbole\footnote{Research of the first and third authors partially funded by NSF grant number DMS-1004624.}\\
Department of Mathematics and Statistics\\
East Tennessee State University\\
Johnson City, TN 37614, USA \\
\href{godbolea@etsu.edu}{\tt godbolea@etsu.edu}
\ \\
Adam Goyt\\ 
Department of Mathematics\\
Minnesota State University\\
Moorhead, MN 56563 USA\\
\href{mailto:goytadam@mnstate.edu}{\tt goytadam@mnstate.edu}
\ \\
Jennifer Herdan\footnotemark[1]\\
Department of Mathematics and Statistics \\
Winona State University\\
Winona, MN 55987, USA
\ \\
Lara Pudwell  \\
Department of Mathematics and Computer Science\\
Valparaiso University\\
Valparaiso, IN 46383, USA\\
\href{mailto:Lara.Pudwell@valpo.edu}{\tt Lara.Pudwell@valpo.edu}\\
\end{center}

\begin{abstract}
In this paper we consider the enumeration of ordered set partitions avoiding a permutation pattern of length 2 or 3.  We provide an exact enumeration for avoiding the permutation 12.  We also give exact enumeration for ordered partitions with 3 blocks and ordered partitions with n-1 blocks avoiding a permutation of length 3. We use enumeration schemes to recursively enumerate 123-avoiding ordered partitions with any block sizes.  Finally, we give some asymptotic results for the growth rates of the number of ordered set partitions avoiding a single pattern; including a Stanley-Wilf type that exhibits existence of such growth rates.
\end{abstract}

Keywords: ordered set partitions, pattern avoidance, permutations, enumeration schemes, growth rates

AMS Subject Classification: 05A05, 05A18, 05A16

\section{Introduction}\label{S:Intro}

Pattern avoidance in permutations was first introduced by Knuth \cite{K73}, and continues to be an active area of research today.  Let $\mathcal{S}_n$ denote the set of permutations of length $n$, and consider $\pi \in \mathcal{S}_n$ and $\rho \in \mathcal{S}_m$.  We say $\pi$ \emph{contains} $\rho$ as a pattern if there exist indices $1 \leq i_1 < i_2 < \cdots < i_{m-1} < i_m \leq n$ such that $\pi_{i_a} \leq \pi_{i_b}$ if and only if $\rho_a \leq \rho_b$.  In this case, we say that $\pi_{i_1}\pi_{i_2}\cdots\pi_{i_m}$ is \emph{order-isomorphic} to $\rho$.  Otherwise, we say $\pi$ \emph{avoids} $\rho$.  Further, let $\mathcal{S}_n(\rho)$ denote the set of permutations of length $n$ that avoid $\rho$, and let $\text{s}_n(\rho) = \left|\mathcal{S}_n(\rho)\right|$.  It is straightforward to see that $\text{s}_n(12) = 1$ for $n \geq 0$ because the only permutation of length $n$ that avoids 12 is the decreasing permutation.  It is also well-known that given any $\rho \in \mathcal{S}_3$, $\text{s}_n(\rho)=C_n$ where $C_n = \dfrac{\binom{2n}{n}}{n+1}$ is the $n$th Catalan number \cite{SS85}.

Pattern avoidance has been studied in contexts other than permutations.  In particular, the notion of pattern-avoidance in set partitions was introduced by Klazar \cite{K96}, with further work done by Klazar, Goyt, and Sagan \cite{G08, GS09, K00, K00b, Sagan10}.  More recently, Goyt and Pudwell introduced the notion of colored set partitions and considered three distinct types of pattern avoidance in this context \cite{GPTBA1, GPTBA2}.  In this paper, we consider a definition of pattern avoidance most closely related to that of \cite{GPTBA2}. 

A partition $p$ of the set $S \subseteq \mathbb{Z}$, written $p \vdash S$, is a family of nonempty, pairwise disjoint subsets $B_1, B_2, \dots , B_k$ of $S$ called \emph{blocks} such that $\displaystyle{\cup_{i=1}^k B_i = S}$.  We write $p = B_1/B_2/\dots/B_k$ and define the {\it length} of $p$, written $\ell(p)$, to be the number of blocks.  Note that because $B_1, \dots , B_k$ are sets, the order of elements within a block does not matter; for convenience we will write elements of a block in increasing order.  We are particularly interested in the set of \emph{ordered partitions of $[n]=\{1,\dots , n\}$ into $k$ blocks}, written $\mathcal{OP}_{n,k}$, which is the set of partitions $p$ such that $p \vdash [n]$, $\ell(p)=k$, and where the order of blocks is important.  For example $13/2/4$ and $4/13/2$ are two distinct members of $\mathcal{OP}_{4,3}$.  In the sequel, we also let $\displaystyle{\mathcal{OP}_n = \cup_{k=1}^n \mathcal{OP}_{n,k}}$, and let $\mathcal{OP}_{[b_1,\dots,b_k]}$ be the set of ordered partitions $p$ such that $p \vdash [b_1 + \cdots + b_k]$, $\ell(p)=k$ and $\left|B_i\right| = b_i$ for $1 \leq i \leq k$.  Similarly, we let $\text{op}_{n,k} = \left|\mathcal{OP}_{n,k}\right|$, $\text{op}_{n} = \left|\mathcal{OP}_{n}\right|$, and $\text{op}_{[b_1,\dots,b_k]} = \left|\mathcal{OP}_{[b_1,\dots,b_k]}\right|$

Given a partition $p \vdash [n]$, and a permutation $\rho \in \mathcal{S}_m$, we say that $p$ contains $\rho$ if there exist blocks $B_{i_1}, \dots, B_{i_m}$ where $i_1 < i_2 < \cdots < i_m$, and there exists $b_j \in B_{i_j}$ such that $b_{1}\cdots b_m$ is order-isomorphic to $\rho$.  For example, $14/56/2/3 \in \mathcal{OP}_{6,4}$ contains the pattern $\rho = 312$, as evidenced by $b_1=4$, $b_2=2$, and $b_3=3$.

Avoidance in ordered partitions is attractive in that three special cases are directly related to other known enumerative results.  First, note that for any $\rho \in \mathcal{S}_m$, we have $\text{s}_n(\rho) = \text{op}_{[\underbrace{1,\dots,1}_{n}]}(\rho)$.  That is, a permutation is equivalent to an ordered partition where all blocks are of size 1.  Second,  we will see in the next section that $\mathcal{OP}_{n,k}(12)$ is in bijection with the set of integer compositions of $n$ into $k$ parts.  Finally, the definition of avoidance for ordered partitions described above corresponds to the pattern-type avoidance detailed in \cite{GPTBA2} if we considered unordered partitions where all elements in block $i$ are given color $c_i$, and avoid the colored pattern $\rho_1^1\rho_2^2\cdots \rho_m^m$ rather than pattern $\rho$.

In Section \ref{S:simple}, we give closed formulas for $\text{op}_{n,k}(12)$, $\text{op}_{n,3}(123)$, and $\text{op}_{n,3}(132)$.  In Section \ref{S:n-1}, we give closed formulas for $\text{op}_{n,n-1}(123)$ and generalize this to give a closed form for $\text{op}_{[b_1,b_2,\dots,b_k]}$ where $b_i=1$ for at least $k-1$ values of $i$.  In Section \ref{S:bij}, we give a bijective proof that $\text{op}_{[b_1,\dots,b_k]}(123) = \text{op}_{[b_1,\dots,b_k]}(132)$ for any list of positive integers $b_1, \dots, b_k$, settling the question of whether $\text{op}_{[b_1,\dots,b_k]}(\rho_1) = \text{op}_{[b_1,\dots,b_k]}(\rho_2)$ for any patterns $\rho_1, \rho_2 \in \mathcal{S}_3$.  In Section \ref{S:enumerate}, we adapt the enumeration schemes of Zeilberger \cite{Z98}, Vatter \cite{V08}, and Pudwell \cite{P10} to enumerate the set $\mathcal{OP}_{[b_1,\dots,b_k]}(123)$ for any block sizes $b_1,\dots,b_k$.  In Section 6 we investigate the case of $\text{op}_{[2,2,...2]}(123)$ and continue with a monotonicity result in Section 7.  Then, in Section 8, we prove that a Stanley-Wilf limit exists as $n$ tends to infinity for all the sequences $\text{op}_{n,k}(\rho)$ where $k$ and $\rho$ are held fixed.  We end with some open questions.

\section{A few simple cases}\label{S:simple}

In this section, we consider a few special cases of the pattern avoidance problem introduced in the introduction.  In particular, we consider $\text{op}_{n,k}(12)$, $\text{op}_{n,3}(123)$, $\text{op}_{n,3}(132)$, and $\text{op}_{n,n-1}(132)$.  The simplest of these cases is that of avoiding the pattern $\rho=12$.

\begin{theorem}
$$\operatorname{op}_{n,k}(12) = \binom{n-1}{k-1}$$
\end{theorem}

\begin{proof}
Notice that a member of $\mathcal{OP}_{n,k}$ avoids the pattern $\rho=12$ if and only if for each block $B_i$, all elements of each block $B_j$ where $j>i$ are strictly less than all elements of block $B_i$.  Once we know the sizes of blocks $B_1, B_2,\dots , B_k$, the 12-avoiding partition is determined: the largest $\left|B_1\right|$ elements are in block 1, the next largest $\left|B_2\right|$ elements are in block 2, and so on.  Thus, $\text{op}_{n,k}(12)$ is merely the number of integer compositions of $n$ into $k$ parts, which is well known to be $\binom{n-1}{k-1}$.
\end{proof}

Note that a similar argument shows that $\text{op}_{n,k}(21) = \binom{n-1}{k-1}$.  In fact, as with pattern-avoiding permutations, a few natural symmetries simplify our work.  Given permutation pattern $\rho=\rho_1\cdots\rho_m$, define the \emph{reversal} of $\rho$, written $\rho^r$, to be $\rho^r = \rho_m\rho_{m-1}\cdots \rho_2\rho_1$, and define the \emph{complement} of $\rho$, written $\rho^c$, to be $\rho^c = (m+1-\rho_1)(m+1-\rho_2)\cdots(m+1-\rho_m)$.  Since complement and reversal both provide involutions on the set of permutations and on the set $\mathcal{OP}_{n,k}$, we have that $\text{op}_{n,k}(\rho) = \text{op}_{n,k}(\rho^r) = \text{op}_{n,k}(\rho^c)=\text{op}_{n,k}((\rho^r)^c)$.  Unlike the case of pattern-avoiding permutations, we no longer have a well-defined notion of inverse, so reversal and complementation are the only natural symmetries of which we may take advantage.

For the case of patterns of length 3, we see that $123^r = 321$, $132^r = 231$, $231^c = 213$, and $213^r=312$, so we have that $\text{op}_{n,k}(123)=\text{op}_{n,k}(321)$ and $\text{op}_{n,k}(132)=\text{op}_{n,k}(231)=\text{op}_{n,k}(213)=\text{op}_{n,k}(312)$.  In Section \ref{S:bij} we will show that $\text{op}_{n,k}(123) = \text{op}_{n,k}(132)$ by demonstrating the stronger claim that $\text{op}_{[b_1,\dots,b_k]}(123)=\text{op}_{[b_1,\dots,b_k]}(132)$ for any choice of block sizes $b_1, b_2, \dots, b_k$.  For now, though, we consider the special case where $k=3$.

\begin{theorem}
$$\operatorname{op}_{n,3}(123) = \left(\dfrac{n^2}{8}+\dfrac{3n}{8}-2\right)2^n+3$$
\label{T:123_3blocks}
\end{theorem}

\begin{proof}
It is easy to see that  
\[\text{op}_{n,3}(123)=\sum_{a=1}^{n-2}\sum_{b=1}^{n-1-a}\sum_{\ell=1}^{n-a+1}\sum_{i=\max\{0,\ell-1-c\}}^{\min\{\ell-1,b\}}{{\ell-1}\choose{i}}{{n-\ell}\choose{a-1}},\]
where $a,b, c=n-a-b$ represent the sizes of blocks 1, 2, and 3 respectively, $\ell$ is the smallest number in block 1, and $i$ represents how many numbers smaller than $\ell$ are in block 2.  Using standard algebraic techniques the above sum above may be reduced to $\left(\dfrac{n^2}{8}+\dfrac{3n}{8}-2\right)2^n+3$.\end{proof}

Next, we consider $\text{op}_{n,3}(132)$.  We will show why $\text{op}_{n,3}(123)=\text{op}_{n,3}(132)$ via a bijection in the next section, but first, we consider how to count members of $\mathcal{OP}_{n,3}(132)$ directly.

\begin{theorem}
$$\operatorname{op}_{n,3}(132) = \left(\dfrac{n^2}{8}+\dfrac{3n}{8}-2\right)2^n+3$$
\label{T:132_3blocks}
\end{theorem}

\begin{proof}
Consider $p \in \mathcal{OP}_{n,3}(132)$.  We have two cases: either $p$ avoids the pattern $12$ in the first two blocks, or $p$ contains the pattern $12$ in the first two blocks.

In the first case, there is no restriction on the values of elements in the third block, so we may choose $a$ elements ($1 \leq a \leq n-2$) to be members of $B_3$.  Since we assume that there is no $12$ pattern in the first two blocks, once we choose the number $b$ ($1 \leq b \leq n-a-1$) of elements of $B_2$, we know the smallest $b$ elements of $B_1$ and $B_2$ are in $B_2$, and the remaining elements are in $B_1$.  Thus, there are $\displaystyle{\sum_{a=1}^{n-2}\sum_{b=1}^{n-a-1}\binom{n}{a}}$ possible partitions of this form.

The second case is more complicated.  Since we know there is a 12 pattern in the first two blocks, let $i$ be the smallest element of $B_1$ that participates in a 12 pattern and let $j$ be the largest element of $B_2$ that participates in a $12$ pattern.  By definition of $i$ and $j$, we know that there are no elements of $B_1$ smaller than $i$ and there are no elements of $B_2$ larger than $j$.  Further, there are no elements of $B_3$ that are both greater than $i$ and less than $j$.  Let $a_1$ be the number of elements of $B_2$ that are greater than $i$ but less than $j$, and let $a_2$ be the number of elements of $B_2$ that are less than $i$.  Further, let $b$ be the number of elements of $B_1$ that are greater than $j$.  Once we have determined, $i$, $j$, $a_1$, $a_2$, and $b$, it remains to choose which $a_1+a_2$ elements appear in $B_2$ and which $b$ elements appear in $B_1$.  This yields $\displaystyle{\sum_{i=1}^{n-1}\sum_{j=i+1}^n \sum_{a_1=0}^{j-i-1}\sum_{a_2=0}^{i-1}\sum_{b=0}^{n-j} \binom{j-i-1}{a_1}\binom{i-1}{a_2}\binom{n-j}{b}-\sum_{i=1}^{n-1}\sum_{j=i+1}^{n} \sum_{a_1=0}^{j-i-1} \binom{j-i-1}{a_1}}$ possible ordered partitions.

Together, we have that 

\begin{align*}
\text{op}_{n,3}(132) &= \sum_{a=1}^{n-2}\sum_{b=1}^{n-a-1}\binom{n}{a}
\\ &+ \sum_{i=1}^{n-1}\sum_{j=i+1}^n \sum_{a_1=0}^{j-i-1}\sum_{a_2=0}^{i-1}\sum_{b=0}^{n-j} \binom{j-i-1}{a_1}\binom{i-1}{a_2}\binom{n-j}{b}\\&-\sum_{i=1}^{n-1}\sum_{j=i+1}^{n} \sum_{a_1=0}^{j-i-1} \binom{j-i-1}{a_1},
\end{align*}
and through standard algebraic manipulation, these sums simplify to $$\text{op}_{n,3}(132) = \left(\dfrac{n^2}{8}+\dfrac{3n}{8}-2\right)2^n+3,$$ as desired.
\end{proof}

\section{The Case of $n-1$ Blocks}\label{S:n-1}

Notice that having $k=3$ blocks is the minimum non-trivial number of blocks to consider when avoiding a pattern of length 3.  Similarly, as mentioned above, $k=n$ blocks is equivalent to considering pattern-avoiding permutations, so the maximum non-trivial number of blocks to consider is $k=n-1$.  In this section we determine $\text{op}_{n,n-1}(123)$, a result that we believe could be the starting point of much new work.  

\begin{theorem} \label{prodform} For $n\geq1$,
$$\text{op}_{n,n-1}(123) = \dfrac{3(n-1)^2\binom{2n-2}{n-1}}{n(n+1)}.$$
\end{theorem}

Theorem \ref{prodform} is a corollary of Theorem \ref{recform}, which was predicted based on computation data and Zeilberger's Maple package \emph{Findrec}~\cite{z08A}.

\begin{theorem}\label{recform} We have
$$\text{op}_{n,n-1}(123) = \begin{cases}
1&\mbox{ if }n=2\mbox{, and}\\
\dfrac{(4n-6)(n-1)^2}{(n-2)^2(n+1)}\text{op}_{n-1,n-2}(123)&n>2
\end{cases}.$$
\end{theorem}

We will prove Theorem \ref{recform} by a series of lemmas.  In particular:

\begin{lemma}\label{blockbijection}
If the list $[c_1,\dots,c_k]$ is a permutation of the list $[b_1,\dots, b_k]$ then $\text{op}_{[b_1,\dots, b_k]}(123) = \text{op}_{[c_1,\dots, c_k]}(123)$.
\end{lemma}

By Lemma \ref{blockbijection}, then we have the following Corollary.

\begin{corollary}\label{ntimes} For $n\geq 2$, 
$\text{op}_{n,n-1}(123) = (n-1)\text{op}_{[2,\footnotesize\underbrace{1,\dots,1}_{n-2}]}(123)$.
\end{corollary}

It then remains to compute $\text{op}_{[2,\underbrace{1,\dots,1}_{n-2}]}(123)$.  We will show that

\begin{lemma}\label{twothenones} For $n\geq 2$,
$\text{op}_{[2,\footnotesize\underbrace{ 1,\dots,1}_{n-2}]}(123) = \dfrac{3(n-1)\binom{2n-2}{n-1}}{n(n+1)}$.
\end{lemma}

The proof of Lemma \ref{twothenones} relies on the following Lemma.

\begin{lemma}\label{catalantriangle}
The number of 123-avoiding permutations of length $n$ that begin with $i$ is given by $c_{n,i} = \dfrac{(n-2+i)!(n-i+1)}{(i-1)!n!}$.
\end{lemma}

\begin{proof}
We will prove this by showing that the number of $132$-avoiding permutations of $[n]$ that begin with $i$ is $c_{n,i}$.  Since the classic bijection between $132$-avoiders and $123$-avoiders given by Simion and Schmidt~\cite{SS85} preserves the first element of the permutation, we will have that the number of $123$-avoiding permutations of $[n]$ beginning with $i$ is $c_{n,i}$.

Let $T_{n,k}$ be the entry in the $n^{th}$ row and $k^{th}$ column of the Catalan Triangle (see OEIS A009766~\cite{OEIS}).  We will prove that $c_{n,i}=T_{n+1,i+1}=\dfrac{(n-2+i)!(n-i+1)}{(i-1)!n!}$.  Callan~\cite{C02} shows that if $C_j$ is the $j^{th}$ Catalan number then $$T_{n,k}=\sum_{j=1}^kC_{j-1}T_{n-j,k-j+1},$$ for $n\geq k$, with $T_{n,0}=1$ for $n\geq0$ and $T_{n,k}=0$ if $k>n$.  

We will show that $c_{n,i}$ satisfies the same recursion.  Let $p\in S_n(132)$ begin with 1, then $p=12\cdots n$.  Thus, $c_{n,1}=1$ for $n\geq 1$.  If $i>n$ then there can be no permutation of $[n]$ beginning with $i$.  

Assume now that $p\in S_n(132)$ and begins with $i$.  Suppose that $n$ is in the $j^{th}$ position.  To avoid $132$ we must have that the $j-1$ largest elements, which must include $i$, appear before $n$, the remaining elements appear after, and each forms a $132$-avoiding permutation.  Thus, since $p$ must begin with $i$, we must have that $n-i+1\leq j\leq n$.  The last $n-j$ elements must form a $132$-avoiding permutation, and there are $C_{n-j}$ such permutations.  

The first $j-1$ elements must form a $132$-avoiding permutation beginning with $i$.  There are $c_{j-1,j-n+i}$ such permutations because $i$ becomes the $(j-n+i)^{th}$ largest element.  

Summing over appropriate values of $j$ we obtain $$c_{n,i}=\sum_{j=n-i+1}^nc_{j-1,j-n+i}C_{n-j}.$$  Rewriting this sum gives $$c_{n,i}=\sum_{j=1}^ic_{n-j,i-j+1}C_{j-1}.$$  

We conclude that $c_{n,i}=T_{n+1,i+1}$ for $n\geq i\geq 1$.  \end{proof}

Now, we prove the other lemmas above.

\begin{proof}[Proof of Lemma \ref{blockbijection}]
We will actually show that $\text{op}_{[b_1,\dots,b_i, b_{i+1},\dots, b_k]}(123) = \text{op}_{[b_1,\dots,b_{i+1}, b_{i},\dots, b_k]}(123)$ for any $1 \leq i \leq k-1$ and for any block sizes $b_1, b_2, \dots, b_k \geq 1$.  Since any permutation of $[b_1, \dots, b_k]$ can be obtained by adjacent transpositions, this suffices to prove Lemma \ref{blockbijection}.

Consider $\pi \in \mathcal{OP}_{[b_1,\dots,b_i, b_{i+1},\dots, b_k]}(123)$.  We construct $f(\pi) \in \mathcal{OP}_{[b_1,\dots,b_{i+1}, b_{i},\dots, b_k]}(123)$ with the following observations.  For each member of $j \in B_{i+1}$ exactly one of the following is true:
\begin{enumerate}
\item $j$ is not part of a 12 pattern involving an element not in $B_i$.
\item $j$ plays the role of a 1 in a 12 pattern involving an element not in $B_i$.  We will call an occurrence involving $B_\ell$, where $\ell >i+1$ is the minimal value for which this happens, a \emph{critical occurrence}.
\item $j$ plays the role of a 2 in a 12 pattern involving an element not in $B_i$.  We will call an occurrence involving $B_\ell$, where $\ell \leq i-1$ is the maximal value for which this happens, a \emph{critical occurrence}.
\end{enumerate}

Clearly 1 and 2 or 1 and 3 cannot occur simultaneously.  If 2 and 3 occur simultaneously then there is a copy of 123 in $\pi$.  

Now, there is a unique way to sort the elements of $B_i \cup B_{i+1}$ into blocks of size $b_{i+1}$ and $b_i$ so that the resulting partition avoids 123 and the critical occurrences of 12 remain occurrences of 12.  Let $\widehat{B_i}$ and $\widehat{B_{i+1}}$ be blocks $i$ and $i+1$ in $f(\pi)$, as defined below.

First the elements of $B_{i+1}$ that are not involved in a critical 12-pattern are placed into $\widehat{B_i}$.  If there are $a$ elements of $B_{i+1}$ playing the role of a 1 in a 12-pattern, then there is some number in a later block that plays the role of 2 in all critical occurrences of 12 with each of these $a$ elements. The largest $a$ elements of $B_{i} \cup B_{i+1}$ that are less than $j$ are placed in $\widehat{B_i}$.  If there are $b$ elements of $B_{i+1}$ playing the role of a 2 in a 12-pattern, then the largest $b$ elements of $B_i \cup B_{i+1}$ are placed in $\widehat{B_i}$.

All other elements of $B_i \cup B_{i+1}$ are placed into $\widehat{B_{i+1}}$.

This is indeed the unique way of rearranging these elements so that we swap the sizes of blocks $i$ and $i+1$, maintain all critical occurrences of 12, and still avoid 123.  First of all, there must be $a$ elements less than $j$ moved to $\widehat{B_i}$, but if they are not the largest $a$ elements less than $j$, we form a 123 pattern using one of these elements, a larger element less than $j$ in $\widehat{B_{i+1}}$, and $j$.  Also, there must be $b$ elements in $\widehat{B_i}$ that play the role of 2 in critical occurrences of 12.  If these are not the largest possible elements in $B_i \cup B_{i+1}$, then we create a 123 pattern using the 1 from the critical occurrence, one of these element from $\widehat{B_i}$, and a larger element in $\widehat{B_{i+1}}$.

Thus, $f(\pi)$ consists of leaving all blocks other than $B_i$ and $B_{i+1}$ unchanged, and rearranging $B_i$ and $B_{i+1}$ as described above.  

\end{proof}

For example, consider the 123-avoiding partition $\pi = 5/37/146/2$.  Let $i=2$.  We wish find a partition with block sizes $b_1=1$, $b_2=3$, $b_3=2$, and $b_4=1$ that avoids 123.  Notice that of the three elements in $B_3 = \{1,4,6\}$.  $1$ is a 1 in the 12- pattern 1/2, 4 is not involved in a 12-pattern outside of blocks $B_2$ and $B_3$, and 6 is a 2 in the 12-pattern 5/6.  Thus, $f(5/37/146/2)  = 5/147/36/2$.

\begin{proof}[Proof of Lemma \ref{twothenones}]
To construct a 123-avoiding partition where the first block has size 2 and all other blocks have size 1, we may begin with a 123-avoiding permutation of length $n-1$ that begins with $i$.  By Lemma \ref{catalantriangle} there are $c_{n,i} = \dfrac{(n-2+i)!(n-i+1)}{(i-1)!n!}$ such permutations.

Then, we insert an element larger than $i$ into the first block.  (Here, inserting $j$ means that all integers in the permutation greater than or equal to $j$ are incremented by 1, and all entries less than $j$ remain the same.)  This new ordered partition certainly avoids 123, since the new element being involved in a 123 pattern means that $i$ would have been involved in a 123-pattern, which contradicts that we began with a 123-avoiding permutation.

If the permutation begins with $i$, then there are $(n-i)$ possible numbers to insert above $i$ to obtain an ordered partition of $[n]$.  Summing over all possible values for $i$, we obtain

$\displaystyle\text{op}_{[2,\underbrace{1,\dots,1}_{n-2}]}(123) = \sum_{i=1}^{n-1} (n-i) c_{n-1,i} = \sum_{i=1}^{n-1} (n-i) \dfrac{(n-3+i)!(n-i)}{(i-1)!(n-1)!} = \dfrac{3(n-1)\binom{2n-2}{n-1}}{n(n+1)}.$

\end{proof}

Now, combining Lemma \ref{twothenones} with Corollary \ref{ntimes} gives us

$$\text{op}_{n,n-1}(123) = (n-1)\dfrac{3(n-1)\binom{2n-2}{n-1}}{n(n+1)} = \dfrac{3(n-1)^2\binom{2n-2}{n-1}}{n(n+1)},$$ which proves Theorem \ref{prodform}.

Note that when $n=2$, this equation simplifies to

$$\text{op}_{2,1}(123) =  \dfrac{3(2-1)^2\binom{2\cdot 2-2}{2-1}}{2(2+1)} = \dfrac{3 \binom{2}{1}}{2\cdot 3} = 1.$$

And when $n>2$, 

$$\dfrac{(4n-6)(n-1)^2}{(n-2)^2(n+1)}\text{op}_{n-1,n-2}(123) = \dfrac{(4n-6)(n-1)^2}{(n-2)^2(n+1)} \cdot \dfrac{3 \cdot (n-2)^2 \binom{2n-4}{n-2}}{(n-1)n}$$

$$= \dfrac{3(4n-6)(n-1) \binom{2n-4}{n-2}}{(n+1)n}=\dfrac{3(n-1)^2 \binom{2n-2}{n-1}}{(n+1)n},$$

which confirms Theorem \ref{recform}.

Note that via algebraic manipulation, our result for $\text{op}_{[2,\underbrace{1,\dots,1}_{n-2}]}(123)$ can be written as

$\text{op}_{[2,\underbrace{1,\dots,1}_{n-2}]}(123) = \dfrac{3\binom{2n-2}{n-2}}{(n+1)}$.

We may generalize to obtain the following:

$\displaystyle\text{op}_{[3,\underbrace{1,\dots,1}_{n-3}]}(123) = \sum_{i=1}^{n-2} \binom{n-i}{2} c_{n-2,i} = \dfrac{4\binom{2n-3}{n-3}}{(n+1)}$, and

$\displaystyle\text{op}_{[4,\underbrace{1,\dots,1}_{n-4}]}(123) = \sum_{i=1}^{n-3} \binom{n-i}{3} c_{n-3,i} = \dfrac{5\binom{2n-4}{n-4}}{(n+1)}$.

In general, we have that 

$\displaystyle\text{op}_{[p,\underbrace{1,\dots,1}_{n-p}]}(123) = \sum_{i=1}^{n-(p-1)} \binom{n-i}{p-1} c_{n-(p-1),i} = \dfrac{(p+1)\binom{2n-p}{n-p}}{(n+1)}$.

Therefore, 

\begin{theorem} For $n> p\geq1$, the number of 123-avoiding ordered partitions of $n$ into $n-p+1$ parts where there is one part of size $p$ and $n-p$ parts of size 1 is given by $$\dfrac{(n-p+1)(p+1)\binom{2n-p}{n-p}}{(n+1)}.$$ \end{theorem}

\section{A Bijection and Pattern Avoidance in Words}\label{S:bij}

As was mentioned in the previous section, the usual symmetries of reversal and complementation are not enough to show that $\text{op}_{[b_1,\dots,b_k]}(123)=\text{op}_{[b_1,\dots,b_k]}(132)$.  We will do this by adapting the familiar bijection of Simion and Schmidt~[11].  To use their bijection we will need a notion of left-to-right minima for set partitions.  Let $p=B_1/B_2/\dots/B_k\in\mathcal{OP}$, then we will say that element $a\in B_i$ is a {\it left-to-right minimum} if $a$ is smaller than every element appearing in blocks $B_j$ for $1\leq j\leq i-1$.  

We will first describe the bijection through example.  Consider the ordered partition $59/38/1267/4$.  This partition avoids 123.  Notice that the left-to-right minima in this partition are 1, 2, 3, 5, and 9.  Also, notice that the other elements in this partition form a decreasing sequence if we place the elements in the same block in decreasing order.  Remove the elements that are not left-to-right minima.  Now we have the partition $59/3/12/\emptyset$, where the second block is missing one element, the third is missing two elements and the last block is missing one element.  We will fill the gaps in the blocks working from left to right by placing the smallest remaining elements that is larger than the smallest left-to-right minimum in the preceding block.  So we would place the element 6 in block 2, obtaining $59/36/12/\emptyset$.  We would then place 4 and 7 in the third block, producing the partition $59/36/1247/\emptyset$.  Finally, we place 8 in the last block producing $59/36/1247/8$.  This partition is 132 avoiding.  

The inverse of this bijection is achieved by placing all of the elements other than the left-to-right minima in descending order.  

\begin{theorem} For $n\geq 1$, $\operatorname{op}_{[b_1,\dots,b_k]}(123)=\operatorname{op}_{[b_1,\dots,b_k]}(132)$.  \end{theorem}
\begin{proof}
Let $\phi:\mathcal{OP}_{[b_1,\dots,b_k]}(123)\rightarrow\mathcal{OP}_{[b_1,\dots,b_k]}(132)$ be as follows.  For $p\in\mathcal{OP}_{[b_1,\dots,b_k]}(123)$, we will construct a corresponding partition $\phi(p)\in\mathcal{OP}_{[b_1,b_2,\dots,b_k]}(132)$. 

First find the left-to-right minima of $p$ and leave them fixed.  Now, remove the other elements of the partition.  Working from left to right fill in the missing entries in each block by placing the smallest remaining element that is larger than the preceding left-to-right minimum.  This new partition will avoid 132, since if a copy of 132 did appear, then one would appear with a left-to-right minimum representing the 1 in the copy of 132.  This would imply that the element representing the 2 was placed in a block after then element representing the 3, which contradicts the prescribed placement of the elements.  Thus, no copy of 132 appears.

The inverse of this construction is to again leave the left-to-right minima in place and place the remaining elements in descending order.  Since the partition will essentially consist of two decreasing sequences, there is no way to form a copy of 123.  \end{proof}

In his thesis~\cite{B98}, Burstein shows that the number of words avoiding the permutation $123$ is the same as the number of words avoiding the permutation $132$ using analytic techniques.  Jel\'{i}nek and Mansour~\cite{JM08} give a bijective proof of the same fact.  It turns out that the bijection above can be used to give another bijective proof of this fact.  We will need to first discuss how words and ordered partitions are related.  

Using the concept of a permutation graph, we define an ordered partition graph to be a permutation graph where we allow more than one entry in a column.  For example, the graph associated to the partition $4/13/256$ is given in Figure 1.

\begin{center}  

$\xy (0,0)*{}; (0,35)*{} **\dir{-};
(0,0)*{}; (20,0)*{} **\dir{-};
(-3,5)*{1};
(-3,10)*{2};
(-3,15)*{3};
(-3,20)*{4};
(-3,25)*{5};
(-3,30)*{6};
(-3,35)*{7};
(0,5)*{}; (20,5)*{} **\dir{-};
(0,10)*{}; (20,10)*{} **\dir{-};
(0,15)*{}; (20,15)*{} **\dir{-};
(0,20)*{}; (20,20)*{} **\dir{-};
(0,25)*{}; (20,25)*{} **\dir{-};
(0,30)*{}; (20,30)*{} **\dir{-};
(0,35)*{}; (20,35)*{} **\dir{-};
(5,0)*{}; (5,35)*{} **\dir{-};
(10,0)*{}; (10,35)*{} **\dir{-};
(15,0)*{}; (15,35)*{} **\dir{-};
(20,0)*{}; (20,35)*{} **\dir{-};
(5,20)*{\bullet};
(10,5)*{\bullet};
(10,15)*{\bullet};
(15,10)*{\bullet};
(15,25)*{\bullet};
(15,30)*{\bullet};
(5,-3)*{1};
(10,-3)*{2};
(15,-3)*{3};
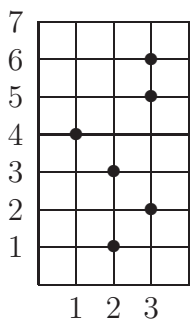
(7,-10)*{\mbox{{\bf Figure 1:} Graph of }4/13/256};
\endxy
$
\end{center}
\medskip

Let $k\geq1$ and $n\geq1$ then the set of words of length $n$ with letters from the alphabet $[k]$ is denoted $[k]^n$.  We say that a word $w\in[k]^n$ {\it contains} $u\in[\ell]^m$ if there are indices $i_1<i_2<\dots<i_m$ such that the word $w_{i_1}w_{i_2}\dots w_{i_m}$ is order isomorphic to $u$.  Otherwise we say that $w$ {\it avoids} $u$.  We let $[k]^n(u)$ be the set of all $w\in[k]^n$ that avoid $u$.  For example the word $232133\in[3]^6$ avoids the word 123.  Since permutations are words in $[n]^n$ without repeated letters, we can consider words that avoid permutations as well.  

We may encode a word using a graph like that above.  In the case of a word, however, we will allow more than one entry in a row.  The word $232133$ has graph.

\medskip
\begin{center}

$\xy (-10,0)*{}; (25,0)*{} **\dir{-};
(-10,0)*{}; (-10,20)*{} **\dir{-};
(-5,-3)*{1};
(0,-3)*{2};
(5,-3)*{3};
(10,-3)*{4};
(15,-3)*{5};
(20,-3)*{6};
(25,-3)*{7};
(-5,0)*{}; (-5,20)*{} **\dir{-};
(0,0)*{}; (0,20)*{} **\dir{-};
(5,0)*{}; (5,20)*{} **\dir{-};
(10,0)*{}; (10,20)*{} **\dir{-};
(15,0)*{}; (15,20)*{} **\dir{-};
(20,0)*{}; (20,20)*{} **\dir{-};
(25,0)*{}; (25,20)*{} **\dir{-};
(-10,5)*{};  (25,5)*{} **\dir{-};
(-10,10)*{}; (25,10)*{} **\dir{-};
(-10,15)*{}; (25,15)*{} **\dir{-};
(-10,20)*{}; (25,20)*{} **\dir{-};
(10,5)*{\bullet};
(-5,10)*{\bullet};
(5,10)*{\bullet};
(0,15)*{\bullet};
(15,15)*{\bullet};
(20,15)*{\bullet};
(-13,5)*{1};
(-13,10)*{2};
(-13,15)*{3};
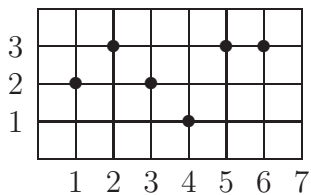
(5,-10)*{\mbox{{\bf Figure 2:} Graph of }232133};
\endxy
$
\end{center}
\medskip

As was mentioned in Section 2, we have three symmetries reversal and complementation which were defined, and inverse which is not a symmetry of ordered partitions, since the inverse operation will not create an ordered partition from any ordered partition with a block with more than one element.  Since the inverse operation is applied to an ordered partition by simply reflecting the graph of the partition in the line $y=x$, we observe that an ordered partition becomes a word when we apply the inverse operation.  Observe that the $232133^i=4/13/256$, where $i$ is the inverse operation.

This gives us that the inverse operation $i$ is a bijection between the set of words $[k]^n$ and the set of ordered partitions with $k$ blocks where we allow blocks to be empty.  For example the word $255332255\in[5]^9$ is mapped to the ordered partition $\emptyset/167/45/\emptyset/2389$, that is $255332255^i=\emptyset/167/45/\emptyset/2389$.  Notice that the first and fourth blocks are empty.  We also observe that the permutations $123$ and $132$ are fixed under the inversion operation.  Thus, if a word $w$ avoids 123 or 132 then so does $w^i$ and similarly for ordered partitions.  

Recall the map $\phi:\mathcal{OP}_{[b_1,\dots,b_k]}(123)\rightarrow\mathcal{OP}_{[b_1,\dots,b_k]}(132)$ from above.  We may extend this map to ordered partitions where we allow some empty blocks, by simply leaving these blocks empty.  For any $w\in [k]^n(123)$, we have that $(\phi(w^i))^i\in[k]^n(132)$.  Let $\psi$ be the operation where we invert a word to form an ordered partition then apply the map $\phi$ above and then invert again to produce a word.  We note that $\psi$ itself is an involution, and hence we have the following theorem:

\begin{theorem} For $n\geq0$, $$\left|[k]^n(123)\right|=\left|[k]^n(132)\right|.$$\end{theorem}

\section{Enumeration schemes for ordered partitions}\label{S:enumerate}

Now that we have investigated ordered partitions avoiding patterns of length 2 and 3, we consider a way to enumerate members of $\mathcal{OP}_{[b_1,\dots,b_k]}(123)$ more generally.  The ideas in this section are largely adapted from the notion of enumeration schemes.  Informally, an enumeration scheme is a system of recurrences that enumerate members of a family of sets.  Such enumeration schemes were first introduced by Zeilberger \cite{Z98} in the context of pattern-avoiding permutations.  Later, Vatter \cite{V08} improved the efficiency of the schemes and Pudwell \cite{P10} generalized enumeration schemes to apply to pattern-avoiding words, that is, permutations of multisets.  One particularly attractive point of enumeration schemes in these contexts have been that the recurrences involved in enumeration schemes can be completely deduced by computer algorithm.  Although such an algorithm has not yet been written for ordered partitions avoiding \emph{any} permutation pattern, we capture many of the same ideas of Zeilberger, Vatter, and Pudwell and adapt them to the current context in order to find a family of recurrences that computes $\text{op}_{[b_1,\dots,b_k]}(123)$ for any list $b_1,\dots ,b_k$ of block sizes.

First, we condition on the patterns formed by the members of the first few blocks of a partition.  To this end, let $$\mathcal{OP}_{[b_1,\dots,b_k]}(123;s) = \{p \in \mathcal{OP}_{[b_1,\dots, b_k]}(123)\vert B_1=s\}$$ and $\text{op}_{[b_1,\dots,b_k]}(123;s) = \left|\mathcal{OP}_{[b_1,\dots,b_k]}(123;s)\right|$. For example, $\text{op}_{[2,1,1]}(123;\{2,4\})=2$ because $\mathcal{OP}_{[2,1,1]}(123;\{2,4\})=\{24/1/3,24/3/1\}$.

Similarly, let 
$$\mathcal{OP}_{[b_1,\dots,b_k]}(123;[s,t]) = \{p \in \mathcal{OP}_{[b_1,\dots, b_k]}(123)\vert B_1=s, B_2=t\}$$ and $\text{op}_{[b_1,\dots,b_k]}(123;[s,t]) = \left|\mathcal{OP}_{[b_1,\dots,b_k]}(123;[s,t])\right|$. For example, $\text{op}_{[2,1,1]}(123;[\{2,4\},\{3\}])=1$ because $\mathcal{OP}_{[2,1,1]}(123;[\{2,4\},\{3\}])=\{24/3/1\}$.

Notice that we immediately have $\text{op}_{[b_1]}(123;\{1,2,\dots,b_1\})=1$ and $\text{op}_{[b_1,b_2]}(123;s)=1$ for any set $s \subseteq [n]$ where $\left|s\right|=b_1$.  More generally, we also see that
$$\mathcal{OP}_{[b_1,\dots,b_k]}(123) = \bigcup_{s \subseteq [n], \left|s\right|=b_1} \mathcal{OP}_{[b_1,\dots,b_k]}(123;s), \text{  and}$$
$$\mathcal{OP}_{[b_1,\dots,b_k]}(123;s) = \bigcup_{t \subseteq [n] \setminus s, \left|t\right|=b_2} \mathcal{OP}_{[b_1,\dots,b_k]}(123;[s,t]).$$  In addition, since the sets on the right hand side of each equation are disjoint, we have
\begin{equation}\text{op}_{[b_1,\dots,b_k]}(123) = \sum_{s \subseteq [n], \left|s\right|=b_1} \text{op}_{[b_1,\dots,b_k]}(123;s), \text{  and}\label{Eq:nullprefix}\end{equation}
\begin{equation}\text{op}_{[b_1,\dots,b_k]}(123;s) = \sum_{t \subseteq [n] \setminus s, \left|t\right|=b_2} \text{op}_{[b_1,\dots,b_k]}(123;[s,t]).\label{Eq:1prefix}\end{equation}

We are interested in $\text{op}_{[b_1,\dots,b_k]}(123)$, which can be written in terms of $\text{op}_{[b_1,\dots,b_k]}(123;s)$. We have an exact value for $\text{op}_{[b_1,\dots,b_k]}(123;s)$ in a couple cases, but otherwise we write $\text{op}_{[b_1,\dots,b_k]}(123;s)$ in terms of $\text{op}_{[b_1,\dots,b_k]}(123;[s,t])$.  It remains to find a recurrence for each $\text{op}_{[b_1,\dots,b_k]}(123;[s,t])$.

Notice that the sets $s$ and $t$ may interact in one of two ways: either all elements in $s$ are larger than all elements of $t$, or they are not.

In the first case, when all elements of $s$ are larger than all elements of $t$, we note that there is a bijection between the sets
$\mathcal{OP}_{[b_1,\dots, b_k]}(123;[s,t])$ and $\mathcal{OP}_{[b_2,\dots,b_k]}(123;t)$.  Consider $p \in \mathcal{OP}_{[b_1,\dots, b_k]}(123;[s,t])$.  Certainly, deleting the entire first block does not produce a new 123 pattern, so after replacing the $i$th smallest element of the remaining ordered partition with $i$, we have obtained a member of $\mathcal{OP}_{[b_2,\dots,b_k]}(123;t)$.  To show that this is indeed a bijection, we also show that no 123-containing partitions would map to $\mathcal{OP}_{[b_2,\dots,b_k]}(123;t)$.  To this end, suppose that $p$ has block sizes $[b_1,\dots, b_k]$, $B_1=s$, $B_2=t$, but $p$ contains a 123 pattern.  If this pattern does not involve elements from $B_1$, then deleting $B_1$ will still produce a partition that contains 123.  If this pattern does involve an element from $B_1$, that element must play the role of ``1'' in the 123 pattern.  However, since all elements of $B_2$ are less than all elements of $B_1$, the roles of ``2'' and ``3'' will be played by elements from $B_3, \dots, B_k$.  This means that even if $B_1$ were deleted, the resulting ordering partition will still contain 123, and we are done.  Since there is a bijection between the sets $\mathcal{OP}_{[b_1,\dots, b_k]}(123;[s,t])$ and $\mathcal{OP}_{[b_2,\dots,b_k]}(123;t)$, we see that \begin{equation}\text{op}_{[b_1,\dots, b_k]}(123;[s,t]) = \text{op}_{[b_2,\dots,b_k]}(123;t).\label{Eq:Rec1}\end{equation}

The second case is more complicated.  If not all elements of $s$ are larger than all elements of $t$, then there exists a 12 pattern within the first two blocks.  Let $i$ be the smallest element of $B_2$ that plays the role of ``2'' in a 12 pattern.  Further let $a$ be the number of elements of $B_1$ that are less than $i$ and let $b$ be the number of elements of $B_2$ that are less than $i$.  Because $i$ plays the role of ``2'' in a 12 pattern, we know that there are no elements larger than $i$ in blocks $B_3, \dots, B_k$.  This means that $\text{op}_{[b_1,\dots,b_k]}(123;[s,t])  = 0$ if $\min(s)<\max(t)$ and $\left([n] \setminus [a+b]\right) \not\subseteq s \cup t$.

Suppose then, that $\min(s)<\max(t)$ and $\left([n] \setminus [a+b]\right) \subseteq s \cup t$.  In this case, $i$ and all elements larger than it cannot be involved in a 123 pattern because these elements all appear in $B_1$ and $B_2$, and there are no elements of $B_3, \dots, B_k$ that are larger than $i$.  We have $\text{op}_{[b_1,\dots,b_k]}(123;[s,t]) = \text{op}_{[a,b,b_3,\dots,b_k]}(123,[s^*,t^*])$ where $s^*$ is the smallest $a$ elements of $s$ and $t^*$ is the smallest $b$ elements of $t$.  Further, since $i$ was the smallest element of $B_2$ involved in a 12 pattern, we now know that all elements of $s^*$ are larger than all elements of $t^*$, so by Equation \ref{Eq:Rec1}, we see that $\text{op}_{[a,b,b_3,\dots,b_k]}(123,[s^*,t^*])=\text{op}_{[b,b_3,\dots,b_k]}(123,t^*)$.

Together, we have that

\begin{equation}
\text{op}_{[b_1,\dots, b_k]}(123;[s,t]) = \begin{cases}
\text{op}_{[b_2,\dots,b_k]}(123;t)& \min(s)>\max(t)\\
0&\min(s)<\max(t) \text{ and }\left([n] \setminus [a+b]\right) \not\subseteq s \cup t\\
\text{op}_{[b,b_3,\dots,b_k]}(123,t^*)& \text{otherwise ($t^*$ is the $b$ least elements of $t$).}
\end{cases}
\label{Eq:Rec2}
\end{equation}

Now, Equations \ref{Eq:nullprefix}, \ref{Eq:1prefix}, \ref{Eq:Rec1} and \ref{Eq:Rec2} together with the base cases $\text{op}_{[b_1]}(123;\{1,2,\dots,b_1\})=1$ and $\text{op}_{[b_1,b_2]}(123;s)=1$ allow us to compute $\text{op}_{[b_1,\dots,b_k]}(123)$ for any list of block sizes.  This family of recurrences has been implemented in the Maple package \texttt{123scheme} can be found on the fourth author's website \url{http://faculty.valpo.edu/lpudwell/maple.html}.  Further, now that we have a way to compute $\text{op}_{[b_1,\dots,b_k]}(123)$, we may use this computation to find $\text{op}_{n,k}(123)$ and $\text{op}_n(123)$ for various values of $n$ and $k$.  Also, by the result of Section \ref{S:bij}, these values will be the same for $\text{op}_{[b_1,\dots,b_k]}(\rho)$, $\text{op}_{n,k}(\rho)$, and $\text{op}_{n}(\rho)$ where $\rho$ is any pattern of length 3.

In the next section, we use the data computed by our enumeration scheme to analyze the case of $\text{op}_{[\underbrace{2,\dots,2}_{k}]}(123)$.

\section{The case of blocks of size 2}\label{S:dominos}

In this section, we consider the case where all partition blocks are of size two.  We may, of course, use the techniques of Section \ref{S:enumerate} to compute $\text{op}_{[\underbrace{2,\dots,2}_{k}]}(123)$.  The enumeration scheme provided in that section efficiently computes $\text{op}_{[\underbrace{2,\dots,2}_{k}]}(123)$, which is sequence A220097 in the Online Encyclopedia of Integer sequences \cite{OEIS}.  However, further analysis with Zeilberger's \texttt{Maple} package \emph{FindRec} \cite{Z08a} predicts that this sequence satisfies this second order linear recurrence: 

\begin{conjecture} $\text{op}_{[\underbrace{2,\dots,2}_{k}]}(123)=$
 $$\dfrac{329k^3-749k^2+514k-96}{2k(2k+1)(7k-9)}\text{op}_{[\underbrace{2,\dots,2}_{k-1}]}(123)+\dfrac{3(14k^3-39k^2+31k-6)}{k(2k+1)(7k-9)}\text{op}_{[\underbrace{2,\dots,2}_{k-2}]}(123).$$ \end{conjecture}
for $k>3$, with $\text{op}_{[2]}(123)=1$ and $\text{op}_{[2,2]}(123)=6$.  It remains an open problem to prove this recurrence describes our sequence in general.

The asymptotic behavior of this sequence resulting from this recurrence can be analyzed using the Birkhoff-Trjitzinsky method \cite{WZ85} (implemented in \emph{AsyRec}~\cite{Z08}), which predicts the behavior to be $\sim .1583\cdot 12^k \dfrac{\left(1-\frac{249}{392k}\right)}{k^{3/2}}$.

Since these experimental techniques predict the sequence $\text{op}_{[\underbrace{2,\dots,2}_{k}]}(123)$ to grow as $12^k$, it is interesting to give an explicit analysis of this case to provide bounds on the asymptotics; the lower bound of $\text{op}_{[\underbrace{2,\dots,2}_{k}]}^{1/k}(123)\ge8$ can be exhibited by elementary means as follows:  If $\pi$ is a 123-avoiding permutation in $S_{2k}$ (of which there are $\frac{{{4k}\choose{2k}}}{2k+1}\sim 16^k$), then $\pi$ corresponds to an {\it ordered} 123-avoiding partition with $k$ consecutive blocks of size 2 each, where the elements in each block are ordered.  Each of these partitions corresponds to $2^k$ {\it ordered} partitions with part sizes 2, it follows that
\[2^k\text{op}_{[\underbrace{2,\dots,2}_{k}]}(123)\ge\frac{{{4k}\choose{2k}}}{2k+1}\sim16^k,\]
which establishes the claim that $\text{op}_{[\underbrace{2,\dots,2}_{k/2}]}^{1/k}(123)\ge\sqrt{8}.$

\section{Monotonicity of $\text{op}_{n,k}(123)$}\label{S:almostperm}

The discussion in Sections \ref{S:simple} and \ref{S:n-1} certainly suggests that $\text {op}_{n,k}(123)$ is not monotone in $k$; for example, for $n=4$ the sequence $\{\text {op}_{n,k}(123):1\le k\le 4\}$ is $\{1,14,27, 14\}$, as seen by elementary counting, results from Section \ref{S:n-1}, and the fact that $C_4=14$.  In a similar vein, Theorem \ref{prodform} shows that $\text {op}_{n,n-1}(123)\sim K\cdot4^n/\sqrt{n}$, whereas it is well known that $\text {op}_{n,n}(123)=C_n\sim K'4^n/n^{3/2}$, a smaller number.  We provide below a first result aimed at understanding the monotonicity (or lack thereof) discussed above:

\begin{theorem}
$\text {op}_{n,4}(123)>\text {op}_{n,3}(123)$ for $n$ sufficiently large.
\end{theorem}
\begin{proof}
We start with a partial injection.  Consider a partition $p$ with three blocks $B_1/B_2/B_3$ with no block being a singleton.  If $n\in B_3$ we map $p$ to
$p':=B_1/B_2/\{n\}/B_3\setminus\{n\}$.  If $n\in B_2$, we map $p$ to $B_1/\{n\}/B_2\setminus\{n\}/B_3$, and if $n\in B_1$ we map $p$ to the 4-partition $p'=\{n\}/B_1\setminus \{n\}/ B_2/B_3$.  It is clear that each of the 4-partitions avoid 123 and that the mapping $p\hookrightarrow p'$ is an injection.  Thus the number  of 3-avoiders with each block having 2 or more elements is less than the number of 4 avoiders with one block being $\{n\}$ and up to one more singleton block.

Simple over-counting shows that the number of 3-partitions (not necessarily avoiding) with at least one singleton block is no more than $3\cdot n\cdot2^{n-1}$ (we choose the singleton in $n$ ways; place it in one of 3 positions; and then choose the two other parts by selecting any subset of the remaining elements).  It is now left to show that the number of 4-avoiders with no singleton block is at least as large as $3\cdot n\cdot2^{n-1}$.  We construct a lower bound on these by the following three-step process:

\begin{enumerate}
\renewcommand{\labelenumi}{\roman{enumi}}
\item make a 123 avoiding 3-partition of $\{3,4,....,n-3\}$ in roughly $\frac{(n-5)^2}{8}\cdot2^{n-5}$ ways (the exact number is given by Theorem 2); 

\item add $n-2$, $2$ and $1$ to the first, second and third blocks respectively;

\item define the first block to be $\{n-1,n\}$.
\end{enumerate}

This completes the proof since  
$3\cdot n\cdot2^{n-1}<\frac{(n-5)^2}{8}\cdot 2^{n-5}$  for $n$ sufficiently large.

\end{proof}
We end with a conjecture that states that monotonicity holds in a certain restricted sense.
\begin{conjecture}
For each fixed $k$, there exists $n_0(k)$ such that for each $\rho\in S_m$ and $n\ge n_0(k)$, $$\operatorname{op}_{n,k+1}(\rho)>\operatorname{op}_{n,k}(\rho)>\cdots>\operatorname{op}_{n,\vert\rho\vert}(\rho).$$
\end{conjecture}

\section{A Stanley-Wilf Type Result}

The Stanley-–Wilf conjecture states that for every permutation $\rho\in\mathcal{S}_m$, there is a constant $C$ such that the number $|S_n(\rho)|$ of permutations of length $n$ which avoid $\rho$ is asymptotic to $C^n$.  The conjecture was first proved by Marcus and Tardos~\cite{MT04}.  We will prove a similar result for ordered partitions.  

Let $\rho\in \mathcal{S}_m$.  Define $\mathcal{OP}*_{n,k}(\rho)$ be the set of ordered $\rho$-avoiding partitions of $[n]$ with $k$ blocks, where some of the blocks may be empty.  Let $\op=\left|\mathcal{OP}*_{n,k}(\rho)\right|$.  We first prove, using Fekete's 1923 lemma for subadditive functions (see \cite{steele}), that a Stanley-Wilf~\cite{MT04} type result holds for $\op$. 
\begin{theorem}
$\displaystyle\lim_{n\to\infty}\opn$ exists as a real number in $[1,\infty)$ for each fixed $k$ and $\rho$.
\end{theorem}
\begin{proof}  Let $\rho\in\mathcal{S}_\ell$ with $\ell\leq k$.  Fix $m,n$, and consider $\pi\in\mathcal{OP}*_{m+n,k}(\rho)$.  We shall show that $\pi$ uniquely determines a pair $(\pi_1,\pi_2)$ where $\pi_1\in \mathcal{OP}^*_{m,k}(\rho)$ and $\pi_2\in \mathcal{OP}*_{n,k}(\rho)$.  The mapping $\pi\hookrightarrow(\pi_1,\pi_2)$ that does this is the one that defines $\pi_1$ as the original partition with only the numbers $\{1,2,\ldots,m\}$ placed in the same $k$ blocks as before, and  with $\pi_2$ defined as the original partition with only the numbers $\{m+1,m+2,\ldots,m+n\}$ placed again in the same $k$ blocks as before, but renumbered as $\{1,2,\ldots,n\}$. To give an example, with $m=5; n=6$, the ordered 321-avoiding block partition $1/\emptyset/3,4,6,10/2,5/7,8/9,11$ decomposes into the two 321-avoiding parts $\pi_1=1/\emptyset/3,4/2,5/\emptyset/\emptyset$ and $\pi_2=\emptyset/\emptyset/1,5/\emptyset/2,3/4,6$.  It follows that 
\[\opmn\le\opm\cdot\op,\]
or that
\[\log\opmn\le\log\opm+\log\op,\]
which shows that the function $\log\op$ is subadditive.  Fekete's lemma (\cite{steele}) thus shows that 
\[\lim_{n\to\infty}\frac{\log\op}{n}=\inf\frac{\log\op}{n};\]
in particular $\lim_{n\to\infty}\frac{\log\op}{n}$ {\it exists} as a number in $I=[-\infty,\infty)$ since the infimum of a non-empty real sequence is always in the interval $I$.  Since in our case $\op\ge1$, it follows that $\lim_{n\to\infty}\frac{\log\op}{n}\in[0,\infty),$ and thus $\lim_{n\to\infty}{\opn}\in[1,\infty),$ as claimed.
\end{proof}

We now show that a Stanley-Wilf limit exists even when blocks are not allowed to be empty, and that the corresponding limit is the same as that for $\op$.  

\begin{theorem}\label{thmSW}
$\lim_{n\to\infty}\oopn$ exists in $[1,\infty)$ for each fixed $k$ and $\rho$.
\end{theorem}

\begin{proof} Let $\rho \in S_m$.  Suppose that $m\leq k$.  We will prove the result by showing that there is a function $\phi(k)$ that satisfies $\frac{\op}{\phi(k)}\le\oop\le\op$, and $\lim_{n\rightarrow\infty}\phi(k)^{1/n}=1$.   

The second inequality is trivial.  Assume that $k\leq n$.  We prove the first inequality by providing an injection from $\mathcal{OP}^*_{n,k}(\rho)$ into $\mathcal{OP}_{n,k}(\rho)\times [0,k]^{2k}$, where $[0,k]^{2k}$ is the number of words with $2k$ letters using the alphabet $[0,k]=\{0,1,\dots,k\}$.  

Let $\pi\in\mathcal{OP}^*_{n,k}(\rho)$, and assume that $\rho$ does not end with $m$.  We construct $(\hat{\pi},w)\in\mathcal{OP}_{n,k}\times [0,k]^{2k}$ in the following way.  First, move all of the empty blocks of $\pi$ to the end of $\pi$ keeping the relative order of the nonempty blocks unchanged, and call this new partition $\pi_0$.  (Note that if $\rho$ ends with $m$ we move the empty blocks to the beginning of $\pi$ and proceed as below.)  Now, the first $k$ letters of $w$ are given by $w_i=j$ if the $i^{th}$ block of $\pi_0$ was the $j^{th}$ block of $\pi$, and $w_i=0$ if the $i^{th}$ block of $\pi_0$ is empty.  Since the relative order of the nonempty blocks of $\pi$ have not changed $\pi_0$ must be $\rho$ avoiding.

Now, suppose there are $a_1$ empty blocks in $\pi_0$.  Remove the $a_1$ largest elements of $\pi_0$ (that is the elements $n-a_1+1$ through $n$), and put one in each of the empty blocks at the end of $\pi_0$ so that they are in increasing order.  Call this new partition $\pi_1$.  

There can be no copies of $\rho$ involving only the first $k-a_1$ blocks since such a copy would have been a copy in $\pi_0$ as well.  If a copy of $\rho$ involves one of the last $a_1$ blocks, then the element in the last block used would have to represent $m$ since it would be the largest of all of the elements from $\pi_1$ used.  This is impossible, since $\rho$ does not end in $m$.  

Suppose there are $a_2$ empty blocks in $\pi_1$.  If $a_2=0$ then we are done, and we let $\pi_1=\hat{\pi}$.  If $a_2>0$ then we take the elements $n-a_1-a_2+1$ through $n-a_1$ out of their blocks and place them in the $a_2$ empty blocks in such a way that the sequence formed by their placement is order isomorphic to the sequence created by the maxima of these blocks in $\pi_0$.  Call this new partition $\pi_2$.  

As before there can be no copies of $\rho$ involving any of the last $a_1$ blocks of $\pi_2$.  Thus, the only possible copies of $\rho$ could be in the first $n-a_1$ blocks of $\pi_2$, and such a copy must involve at least one of the singleton blocks that was an empty block in $\pi_1$.  Such a copy of $\rho$ cannot only involve these formerly empty blocks of $\pi_1$ since this would imply a copy of $\rho$ in $\pi_0$.  Thus, a copy of $\rho$ must involve elements from an original block from $\pi_1$ and some of the formerly empty blocks.  

Suppose such a copy exists, call it $\sigma=\sigma_1\sigma_2\cdots\sigma_m$.  Suppose that $\sigma_{i_1}\sigma_{i_2}\dots\sigma_{i_t}$ is the subsequence of elements from the formerly empty blocks.  Replace each of these entries in $\sigma$ by the maxima of the corresponding blocks in $\pi_0$.  Call this new permutation $\tau$.  Now, $\tau$ is order isomorphic to $\sigma$ since the sequence $\sigma_{i_1}\sigma_{i_2}\dots\sigma_{i_t}$ is order isomorphic to the sequence of maxima of the corresponding blocks in $\pi_0$, and $\min\{\sigma_{i_j}:1\leq j\leq t\}>\max\{\sigma_s:s\notin\{i_j:1\leq j\leq t\}\}$.  Thus, $\tau$ is a copy of $\rho$ in $\pi_0$, which contradicts the fact that $\pi_0$ avoids $\rho$. 

Suppose we reach a partition $\pi_i$ with $a_{i+1}$ empty blocks.  We remove the largest $a_{i+1}$ that are in blocks that have not been empty at any point during the construction.  We fill the empty blocks with these $a_{i+1}$ elements by placing one in each block so that they are order isomorphic to the maxima of the corresponding blocks in $\pi_{i-1}$, thus forming $\pi_{i+1}$.  

We must show that $\pi_{i+1}$ avoids $\rho$.  Suppose $\sigma=\sigma_1\sigma_2\cdots\sigma_m$ is a copy of $\rho$ in $\pi_{i+1}$.  We know that none of the last $a_1$ blocks are involved.  Form a permutation $\tau=\tau_1\tau_2\dots\tau_m$ from $\sigma$ in the following way.  If $\sigma_i$ is in a block that has not yet been emptied then $\tau_i=\sigma_i$.  If $\sigma_i$ is in a block that was empty in $\pi_j$ then $\tau_i$ is the maximum of the corresponding block in $\pi_{j-1}$.  Now, $\tau$ is order isomorphic to $\sigma$ by an argument similar to that for $\pi_2$.  Thus, $\tau$ is a copy of $\rho$ in $\pi_0$.  

Once we reach $\pi_j$ with no empty blocks, we set $\pi_j=\hat{\pi}$.  Such a $\pi_j$ will always be obtained since we assumed that $n\geq k$.  

The final $k$ letters of the word $w$ are as follows.  Let $w_{k+i}=0$ if the $i^{th}$ block of $\hat{\pi}$ has more than one element.  Let $w_{k+i}=j$ if the $i^{th}$ block has one element and that element was in the $j^{th}$ block of $\pi_0$.  The nature of $w$ makes this process invertible and hence we have an injection.  Since $|[0,k]^{2k}|=(1+k)^{2k}$, we have that $\frac{\op}{\phi(k)}\le\oop$ where $\phi(k)=(1+k)^{2k}$.  \end{proof}

An example will certainly help make the previous construction easier to understand.  Suppose that $\rho=132$.  We have that $\pi=8/\emptyset/359/12/\emptyset/46/7\in \mathcal{OP}^*_{9,7}(\rho)$.  In the first step we move the empty blocks to the end, and obtain $\pi_0=8/359/12/46/7/\emptyset/\emptyset$.  The first seven letters of the word $w$ are 1346700 since the last 2 blocks were empty, the second block in $\pi_0$ was the third block in $\pi$, etc.  

Now, we remove 8 and 9 from $\pi_0$ and use them to fill in the empty blocks at the end by placing them in increasing order to obtain, $\pi_1=\emptyset/35/12/46/7/8/9$.  The first block has been emptied by removing the 8.  

We remove 7 from its block in $\pi_1$ and place it in the empty block.  Notice that since there is only one element to place we do not need to worry about placing it in an order isomorphic way.  This gives us $\pi_2=7/35/12/46/\emptyset/8/9$.  The fifth block has been emptied by removing the 7.  

We remove 6 from its block in $\pi_2$ and place it in the empty block obtaining $\pi_3=7/34/12/4/6/8/9$.  No blocks of $\pi_3$ are empty, so we set $\hat{\pi}=\pi_3$.  The last seven letters of the word $w$ are 5004412.  The second and third of these last seven letters are zeros since the second and third blocks of $\hat{\pi}$ have at least two elements each.  The first letter is 5 since 7 was in the fifth block of $\pi_0$ etc.  

This gives us the pair $(7/35/12/4/6/8/9,13467005004412)\in \mathcal{OP}_{9,7}(132)\times [0,k]^{2k}$.  

The proof of Theorem \ref{thmSW} can be substantially simplified if the monotonicity conjecture 2 is proved, since we would have that
$\text{op}^*(n,k) \le\sum_{j=1}^k {k\choose j} \text{op}(n,j),$ where $j$ indicates which of the $k$ blocks are to be non-empty; varying the positions of these yields all possibilities for empty ordered blocks. Thus by monotonicity,
$$\text{op}^*(n,k) \le (2^k-1)\max_{1\le j\le k} \text{op}(n,j) = (2^k-1) \text{op}(n,k),$$ for $n$ sufficiently large, which proves the first inequality in Theorem \ref{thmSW} with $\phi(k)=2^k-1$.

\section{Future work/Open questions}\label{S:future}  

We propose the following open questions for study. 
 
\begin{enumerate}
\renewcommand{\labelenumi}{\roman{enumi}}
\item Enumeration of classes of avoidance numbers $\text {op}_{n,k}(\rho)$, $k\ge4$, is certainly of critical importance, starting with the case $\vert\rho\vert=3$; 

\item We know that $\lim\text {op}_{n,3}^{1/n}(123)=2$, and $\lim\text {op}_{n,n}^{1/n}(123)=4$.  For which $k$ does, e.g., $\lim\text {op}_{n,k}^{1/n}(123)=3$?  Is $\sup_{k=k_n}\lim\text {op}_{n,k}^{1/n}(123)<\infty$?  Is $\sup_{k=k_n}\lim\text {op}_{n,k}^{1/n}(123)=4$? 

\item To what extent is Conjectures 1 and 2 true?  

\item Are the numbers $\text {op}_{n,k}(\rho)$ unimodal for $3\le k\le n$ and where do they attain their maximum?

\end{enumerate}

\end{document}